\documentclass[10pt]{amsart}
\usepackage{amsmath,amsfonts,amssymb,amscd,graphicx,geometry,color,hyperref}

\newtheorem{theorem}{Theorem}[section]
\newtheorem{proposition}[theorem]{Proposition}
\newtheorem{corollary}[theorem]{Corollary}

\newtheorem{lemma}[theorem]{Lemma}

\numberwithin{equation}{section}

\begin{document}
\title{A subexponential vector-valued Bohnenblust-Hille type inequality}

\author[Albuquerque]{N. Albuquerque}
\address{Departamento de Matem\'atica, \newline
\indent Universidade Federal da Para\'iba, \newline
\indent 58.051-900 - Jo\~ao Pessoa, Brazil.}
\email{ngalbqrq@gmail.com}

\author[N\'u\~nez]{D. N\'u\~nez-Alarc\'on}
\address{Departamento de Matem\'atica, \newline
\indent Universidade Federal de Pernambuco, \newline
\indent 50.740-560 - Recife, Brazil.}
\email{danielnunezal@gmail.com}

\author[Serrano]{D. M. Serrano-Rodr\'iguez}
\address{Departamento de Matem\'atica, \newline
\indent Universidade Federal de Pernambuco, \newline
\indent 50.740-560 - Recife, Brazil.}
\email{dmserrano0@gmail.com}


\thanks{N. Albuquerque was supported by CAPES}


\begin{abstract}
Bayart, Pellegrino and Seoane recently proved that the polynomial
Bohnenblust--Hille inequality for complex scalars is subexponential. We show
that a vector valued polynomial Bohnenblust-Hille inequality on complex
Banach lattices is also subexponential for some special cases. Our main
result recovers the best known constants of the classical polynomial
inequality provided in \cite{bps}.
\end{abstract}

\maketitle

%%%%%%%%%%%%%%%%%%%%%%%%%%%%%%%%%%%%%%%%
%%%%%%%%%%%%%%%%%%%%%%%%%%%%%%%%%%%%%%%%
\section{Introduction and preliminaries}
%%%%%%%%%%%%%%%%%%%%%%%%%%%%%%%%%%%%%%%%
%%%%%%%%%%%%%%%%%%%%%%%%%%%%%%%%%%%%%%%%

The Bohnenblust--Hille inequality for complex homogeneous polynomials (\cite{bh}, 1931) asserts that there is a function $\mathcal{D}:\mathbb{N}\rightarrow \lbrack 1,\infty )$ such that no matter how we select a positive integer $N$ and an $m$-homogeneous polynomial $P$ on $\mathbb{C}^{N}$, the $\ell _{\frac{2m}{m+1}}$-norm of the set of coefficients of $P$ is bounded above by $\mathcal{D}(m)$ times the supremum norm of $P$ on the unit polydisc. Having good estimates for $\mathcal{D}(m)$ is crucial for applications (for instance to the determination of the exact asymptotic growth of the Bohr radius). In (\cite{DefAnnals}, 2011) it was proved that
\begin{equation}
\mathcal{D}(m)\leq \left( 1+\frac{1}{m}\right) ^{m-1}\sqrt{m}\left( \sqrt{2}\right)^{m-1} \label{defant55}
\end{equation}
which yields the hypercontractivity of the Bohnenblust-Hille inequality. The best known estimates for $\mathcal{D}(m)$ are due to F. Bayart, D. Pellegrino and J. Seoane (\cite{bps}, 2014) and show that the growth of $\mathcal{D}(m)$ is subexponential. More precisely, in \cite{bps} it is shown that for any $\varepsilon >0$, there is $\kappa >0$ such that 
\begin{equation*}
\mathcal{D}(m)\leq \kappa \left( 1+\varepsilon \right) ^{m},
\end{equation*}
for all positive integers $m$.

\vskip.3cm

In this paper we obtain similar estimates for some special Banach lattices. More specifically, we obtain a subexponential vector-valued Bohnenblust--Hille type inequality that combines the behaviour of $m$-homogeneous polynomial with values in Banach functions spaces, under the action of an $(r;1)$--summing operator. The study of this type inequalities was introduced by A. Defant, M. Maestre and U. Schwarting in \cite{DMS} and it was used as an important technical tool for estimate the multidimensional Bohr radius for such operators. Our approach is inspired on ideas and results from \cite{bps,DMS}.

\vskip.3cm

Let us recall some classical notions from Banach Space Theory. A Banach space $X$ has \emph{cotype} $q$, for $2\leq q<\infty$, if there is a constant $C>0$ such that 
\begin{equation}
\left( \sum_{i=1}^{n}\Vert x_{j}\Vert^{q}\right) ^{\frac{1}{q}}
\leq C\left(\int_{I}\left\Vert \sum_{j=1}^{n}r_{j}(t)x_{j}\right\Vert ^{2}dt\right)^{\frac{1}{2}},  \label{qqq}
\end{equation}
no matter how we select finitely many vectors $x_{1},\dots,x_{n}\in X$, where $I:=[0,1]$ and $r_{i}$ denotes the $i$-th Rademacher function. The smallest of all these constants is denoted by $C_{q}(X)$ and it is called the cotype $q$ constant of $X$. The infimum of the cotypes assumed by $X$ is denoted by $\cot X$. A result due to Kahane generalizes Khinchine's inequality to arbitrary Banach spaces and provides estimates between $L_{p}$ norms of Rademacher means:

\begin{theorem}[Kahane Inequality]
\label{kahane} Let $0<p,q<\infty$. Then there is a constant $\mathrm{K}_{p,q}>0$ for which 
\begin{equation*}
\left( \int_{I} \left\Vert \sum_{k=1}^{n} r_{k}(t) x_{k} \right\Vert^{q} dt\right)^\frac{1}{q}
\leq\mathrm{K}_{p,q} \left( \int_{I}\left\Vert 
\sum_{k=1}^{n} r_{k}(t)x_{k}\right\Vert ^{p} dt \right)^\frac{1}{p}
\end{equation*}
for all Banach spaces $X$ and $x_{1},\dots,x_{n}\in X$.
\end{theorem}

For $p\in\lbrack 1,\infty)$, the \emph{weak $\ell_{p}$-norm} of $x_{1},\dots,x_{N}$ in a Banach space $X$ is defined by
\begin{equation*}
\Vert\left( x_{i}\right) _{i=1}^{N}\Vert_{w,p}
:=\sup_{\Vert x^{\prime}\Vert_{X^{\prime}}\leq1} 
\left( \sum_{i=1}^{N}|x^{\prime}(x_{i})|^{p}\right) ^{\frac{1}{p}}.
\end{equation*}

Throughout this paper we will only consider complex Banach spaces. From now on $X,X_{1},\dots ,X_{m},Y$ stands for Banach spaces and $\mathcal{L}\left( X_{1},\dots ,X_{m};Y\right) $ denotes the Banach space of all (bounded) $m$-linear operators $U:X_{1}\times \cdots \times X_{m}\rightarrow Y$. A linear operator $u:X\rightarrow Y$ is \emph{absolutely summing} if $\Vert (u(x_{j}))_{j=1}^{\infty }\Vert _{\ell _{1}}<\infty $, whenever $\Vert \left( x_{j}\right) _{j=1}^{\infty }\Vert _{w,1}<\infty $. The concept of absolutely summing operator was extended to a general notion of \emph{absolutely $(p;q)$-summing} operators in the 1960's by B. Mitiagin and A. Pe\l czy\'{n}ski \cite{MiP} and A. Pietsch \cite{stu}. A linear operator $u:X\to Y$ is \emph{absolutely $(p;q)$-summing} if $\Vert (u(x_{j}))_{j=1}^{\infty }\Vert _{\ell _{p}}<\infty $ whenever $\Vert \left( x_{j}\right) _{j=1}^{\infty }\Vert _{w,q}<\infty $.

\vskip.3cm

A multilinear approach to absolutely $(p;q)$-summing operators is the concept of \emph{multiple $(q,p)$-summing} o\-pe\-ra\-tors (see, \emph{e.g.}, \cite{matos,david}). For $q,p_{1}\dots ,p_{m}\in \lbrack 1,+\infty ),\,U\in \mathcal{L}\left( X_{1},\dots ,X_{m};Y\right) $ is called \emph{multiple $(q,p_{1},\dots ,p_{m})$-summing} if there exists a constant $C>0$ such that 
\begin{equation*}
\left( \sum_{i_{1},\dots ,i_{m}=1}^{N}
  \left\| U\left(x_{i_{1}}^{(1)},\dots ,x_{i_{m}}^{(m)}\right) \right\|_{Y}^{q}\right)^{\frac{1}{q}}
\leq C\prod_{k=1}^{m}\left\Vert \left( x_{i}^{(k)}\right)_{i=1}^{N}\right\Vert _{w,p_{k}}
\end{equation*}
holds for each finite choice of vectors $x_{i}^{(k)}\in X_{k},\,1\leq i\leq N,\,1\leq k\leq m$. The vecto r space of all multiple $(q,p_{1},\dots ,p_{m})$-summing operators is denoted by $\Pi _{(q,p_{1},\dots,p_{m})}^{m} \left(
X_{1},\dots ,X_{m};Y\right) $. The infimum, $\pi _{\left( q,p_{1},\dots,p_{m}\right) }^{m}(U)$, taken over all possible constants $C$ satisfying the previous inequality defines \linebreak a complete norm in $\Pi_{\left(
q,p_{1},\dots ,p_{m}\right) }^{m}\left( X_{1},\dots ,X_{m};Y\right) $. When $p_{1}=\cdots =p_{m}=p$ we say that the $U$ is multiple $(q,p)$-summing \emph{ \ } and \emph{\ } the notation is changed to $\Pi _{(q,p)}^{m}\left(
X_{1},\dots ,X_{m};Y\right) $ and $\pi _{\left( q,p\right) }^{m}(\cdot )$. When $p=q$ we say that $U$ is multiple $p$-summing, and the notation is simplified to $\Pi _{p}^{m}\left( X_{1},\dots ,X_{m};Y\right) $ and $\pi_{p}^{m}(\cdot)$.

\vskip.3cm

A vector lattice $X$ is a \emph{Banach lattice} when it is a Banach space and has the property that $\left\Vert x\right\Vert <\left\Vert y\right\Vert$, whenever $x,y\in X$ fulfils $\left\vert x\right\vert <\left\vert y\right\vert $. Let us recall some notions from the theory of Banach lattices. A Banach lattice $X$ is $q$-\emph{concave}, $1\leq q<\infty $, if there is a constant $C>0$ such that, regardless of the choice of finitely many $x_{1},\dots ,x_{N}\in X$, we have
\begin{equation*}
\left( \sum_{n=1}^{N}\left\Vert x_{n}\right\Vert ^{q}\right) ^{\frac{1}{q}}
\leq C\left\Vert \left( \sum_{n=1}^{N}\left\vert x_{n}\right\vert^{q}\right)^{\frac{1}{q}}\right\Vert,
\end{equation*}
and the optimal constant $C$ is denoted by $M_{q}(X)$. On the other hand, a Banach lattice $X$ is called $p$-\emph{convex}, $1\leq p<\infty $, if there is a constant $C>0$ such that for every $x_{1},\dots ,x_{N}\in X$ we have
\begin{equation*}
\left\Vert \left( \sum_{n=1}^{N}\left\vert x_{n}\right\vert ^{p}\right) ^{\frac{1}{p}}\right\Vert
\leq C\left( \sum_{n=1}^{N}\left\Vert x_{n}\right\Vert ^{p}\right) ^{\frac{1}{p}},
\end{equation*}
and the optimal constant $C$ is denoted by $M^{p}(X)$. For more details see, \emph{e.g.}, \cite{LindTzaf,OkRiSan}.

\vskip.3cm

In \cite{DMS} the authors provide a vector-valued hypercontractive Bohnenblust--Hille type inequality for Banach lattices that is a strong extension of \cite[Theorem 1]{DefAnnals} and also the main technical tool for the estimates on multidimensional Bohr radii for operators from \cite{DMS}. The precise result is the following:

\begin{theorem}[{Defant, Maestre, Schwarting, \protect\cite[Theorem 5.3]{DMS}}]
\label{dms} Let $Y$ be a $q$-concave Banach Lattice, with $2\leq q<\infty $,
and $v:X\rightarrow Y$ an $(r,1)$-summing operator with $1\leq r<q$. Define 
\begin{equation*}
\rho :=\frac{qmr}{q+(m-1)r}.
\end{equation*}%
Then there is a constant $\mathcal{C}>0$ such that for every $m$-homogeneous
polynomial $P:\ell _{\infty }^{n}\rightarrow Y$, defined by $%
P(z)=\sum_{|\alpha |=m}c_{\alpha }z^{\alpha }$, the following holds 
\begin{equation*}
\left( \sum_{|\alpha |=m}\left\Vert vc_{\alpha }\right\Vert _{Y}^{\rho
}\right) ^{\frac{1}{\rho }}\leq \mathcal{C}^{m}\Vert P\Vert _{\mathbb{D}%
^{n}}.
\end{equation*}
\end{theorem}

A special case of Banach lattices are the Banach function spaces. Throughout
this paper we fix a positive, finite measure space $\left( \Omega ,\Sigma
,\mu \right) $. Let $\text{sim }\Sigma $ denote the vector space of all
complex-valued $\Sigma $-simple functions and $L^{0}\left( \mu \right) $ be
the space of all (equivalence classes of) complex-valued $\Sigma $%
-measurable functions modulo $\mu $-null functions. The space $L^{0}\left(
\mu \right) $ is a complex vector lattice \cite[p. 18]{OkRiSan}. An order
ideal $X\left( \mu \right) \subset L^{0}\left( \mu \right) $ is called a 
\emph{Banach function space} (B.f.s. for short) based on the measure space $%
\left( \Omega ,\Sigma ,\mu \right) $ if $\text{sim }\Sigma \subseteq X\left(
\mu \right) $ and $X\left( \mu \right) $ is equipped with a lattice complete
norm $\left\Vert \cdot \right\Vert _{X\left( \mu \right) }$.

Now we remind the notion of a $p$-th power of a Banach function space as
presented in \cite{OkRiSan}. Suppose that $\left( X\left( \mu \right)
,\left\Vert \cdot \right\Vert _{X\left( \mu \right) }\right) $ is a $p$%
-convex B.f.s. with $p$-convexity constant equals to $1$ (if $M^{p}\left(
X\left( \mu \right) \right) \not=1$, one may renorm de space in order to
obtain it \cite[Proposition 1.d.8]{LindTzaf}), for some $1\leq p\leq 2$. The 
\emph{$p$-th power} is defined as 
\begin{equation*}
X\left( \mu \right) _{[p]}:=\left\{ f\in L^{0}\left( \mu \right) :\left\vert
f\right\vert ^{\frac{1}{p}}\in X\left( \mu \right) \right\}
\end{equation*}%
and%
\begin{equation*}
\left\Vert f\right\Vert _{X\left( \mu \right) _{[p]}}:=\left\Vert \left\vert
f\right\vert ^{\frac{1}{p}}\right\Vert _{X\left( \mu \right) }^{p},\text{ \
\ }f\in X\left( \mu \right) _{[p]}\text{.}
\end{equation*}%
It is well-known that $\left\Vert \cdot \right\Vert _{X\left( \mu \right)
_{[p]}}$ defines a lattice complete norm on the order ideal $X\left( \mu
\right) _{[p]}$ of the vector lattice $L^{0}\left( \mu \right) $, with $%
\text{sim }\Sigma \subseteq X\left( \mu \right) _{[p]}$. In other words, $%
\left( X(\mu )_{[p]},\left\Vert \cdot \right\Vert _{X(\mu )_{[p]}}\right) $
is a B.f.s..

\bigskip

In Section 3 we improve the hypercontractive estimate of Theorem \ref{dms} for the particular case of $2$-concave and $2$-convex B.f.s..

\vskip.3cm

We also need some useful multi-index notation: for positive integers $m,n$,
we set 
\begin{align*}
\mathcal{M}(m,n)& :=\left\{ \mathbf{i}=(i_{1},\dots ,i_{m});\,i_{1},\dots
,i_{m}\in \{1,\dots ,n\}\right\} , \\
\mathcal{J}(m,n)& :=\left\{ \mathbf{i}\in \mathcal{M}(m,n);\,i_{1}\leq
i_{2}\leq \dots \leq i_{m}\right\} ,
\end{align*}%
and for $k=1,\dots ,m$, $\mathcal{P}_{k}(m)$ denotes the set of the subsets
of $\{1,\dots ,m\}$ with cardinality $k$. For $S=\{s_{1},\dots ,s_{k}\}\in 
\mathcal{P}_{k}(m)$, its complement will be $\widehat{S}:=\{1,\dots
,m\}\setminus S$, and $\mathbf{i}_{S}$ shall mean $(i_{s_{1}},\dots
,i_{s_{k}})\in \mathcal{M}(k,n)$. For a multi-indexes $\mathbf{i}\in 
\mathcal{M}(m,n)$, we denote by $|\mathbf{i}|$ the cardinality of the set of
multi-index $\mathbf{j}\in \mathcal{M}(m,n)$ such that there is a
permutation $\sigma $ of $\{1,\dots ,m\}$ such that $i_{\sigma (k)}=j_{k}$,
for every $k=1,\dots ,m$.

\vskip.3cm

The following powerful generalization of the Blei inequality will be crucial
for our estimates.

\begin{lemma}[Bayart, Pellegrino, Seoane]
\label{blei.interp} Let $m,n$ positive integers, $1\leq k\leq m$ and $1\leq
s\leq q$. Then for all scalar matrix $\left( a_{\mathbf{i}}\right) _{\mathbf{%
i}\in\mathcal{M}(m,n)}$,

\begin{equation*}
\left( \sum_{\mathbf{i}\in \mathcal{M}(m,n)}\left\vert a_{\mathbf{i}%
}\right\vert^\frac{msq}{kq+(m-k)s} \right)^\frac{kq+(m-k)s}{msq} \leq
\prod_{S\in \mathcal{P}_k (m)} \left( \sum_{\mathbf{i}_{S}} \left( \sum_{%
\mathbf{i}_{\widehat{S}}} \left\vert a_{\mathbf{i}}\right\vert^q
\right)^\frac sq \right)^{\frac1s \cdot \frac{1}{\binom{m}{k}}}.
\end{equation*}
\end{lemma}

\vskip.3cm

\bigskip

\section{A subpolynomial multilinear Bohnenblust-Hille inequality on Hilbert
lattices}

The multilinear Bohnenblust-Hille inequality was originally proved in 1931
with constants with exponential growth. It was only in 2012 that Pellegrino
and Seoane, based on previous work of Defant, Popa and Schwarting, provided
(surprising) constants with subpolynomial growth (see \cite{popa,diniz,ps}).
In this section we show that for some vector-valued cases the
Bohnenblut-Hille type inequality also has constants with subpolynomial
growth. Our results improve some estimates from \cite{abps.hlw}, in some
special cases.

\vskip.3cm

Thanks to Grothendieck we know that every continuous linear operator from $%
\ell_{1}$ to any Hilbert space is absolutely summing. This result is one of
the several consequences of the famous Grothendieck's inequality, called
\textquotedblleft the fundamental theorem the metric theory of tensor
products\textquotedblright. A Banach space $X$ is said \emph{GT space} if
every continuous linear operator from $X$ to $H$ is absolutely summing (thus
is $p$-summing for $1\leq p<\infty$), where $H$ stands for a Hilbert space.

\vskip.3cm

The next result from \cite[Theorem 4.3]{anss} improves the estimates on
variations of Bohnenblust-Hille inequality introduced in \cite{nps} and \cite%
[Appendix A]{npss}.

\begin{theorem}
\label{BH_variant} Let $t\in \lbrack 1,2)$ and $m>1$. Then 
\begin{equation*}
\left( \sum_{i_{1},\dots,i_{m}=1}^{\infty}\left\vert U\left(
e_{i_{1}},\dots,e_{i_{m}}\right) \right\vert ^{\frac{2tm}{2+(m-1)t}}\right)
^{\frac{2+(m-1)t}{2tm}}\leq\mathrm{C}_{m,t}^{\mathbb{C}}\left\Vert
U\right\Vert ,
\end{equation*}
for all $m$-linear forms $U:c_{0}\times\cdot\cdot\cdot\times c_{0}\rightarrow%
\mathbb{C}$, with 
\begin{equation*}
\mathrm{C}_{m,t}^{\mathbb{C}} \leq\prod_{j=2}^{m}\Gamma\left( 2-\frac {2-t}{%
jt-2t+2}\right) ^{\frac{t(j-2)+2}{2t-2jt}}.
\end{equation*}
\end{theorem}

Theses estimates combined with the Grothendieck's Theorem lead us to improve
the constants of some special cases from \cite[Theorem 6.1]{abps.hlw}:

\begin{theorem}
Let $H$ be a Hilbert space, $r\in \lbrack 1,2)$ and $v:\ell _{1}\rightarrow
H $ a linear operator. Then 
\begin{equation*}
\left( \sum_{i_{1},\dots ,i_{m}=1}^{\infty }\left\Vert vT\left(
e_{i_{1}},...,e_{i_{m}}\right) \right\Vert _{H}^{\frac{2rm}{2+\left(
m-1\right) r}}\right) ^{\frac{2+\left( m-1\right) r}{2rm}}\leq \mathrm{C}%
_{m,r}^{\mathbb{C}}\pi _{\frac{2r(m}{2+(m-1)r}}(v)\left\Vert T\right\Vert ,
\end{equation*}%
for all $m$-linear operators $T:c_{0}\times \cdots \times c_{0}\rightarrow $ 
$\ell _{1}$, with $\mathrm{C}_{m,r}^{\mathbb{C}}$ as in Theorem \ref%
{BH_variant}.
\end{theorem}

\begin{proof}
By Grothendieck's theorem, $v$ is $\rho $-summing, with $\rho :=\frac{2rm}{%
2+(m-1)r}$. So, 
\begin{equation*}
\left( \sum_{i_{1},\dots ,i_{m}=1}^{\infty }\left\Vert vT\left(
e_{i_{1}},...,e_{i_{m}}\right) \right\Vert _{H}^{\rho }\right) ^{\frac{1}{%
\rho }}\leq \pi _{\rho }(v)\sup_{\phi \in B_{\ell _{1}^{\prime }}}\left(
\sum_{i_{1},\dots ,i_{m}=1}^{\infty }\left\vert \phi T\left(
e_{i_{1}},...,e_{i_{m}}\right) \right\vert ^{\rho }\right) ^{\frac{1}{\rho }%
}.
\end{equation*}%
Since the operator $\phi T$ is an $m$-linear form, Theorem \ref{BH_variant}
lead us to 
\begin{equation*}
\sup_{\phi \in B_{\ell _{1}^{\prime }}}\left( \sum_{i_{1},\dots
,i_{m}=1}^{\infty }\left\vert \phi T\left( e_{i_{1}},...,e_{i_{m}}\right)
\right\vert ^{\rho }\right) ^{\frac{1}{\rho }}\leq \mathrm{C}_{m,r}^{\mathbb{%
C}}\left\Vert T\right\Vert .
\end{equation*}%
Therefore, 
\begin{equation*}
\left( \sum_{i_{1},\dots ,i_{m}=1}^{\infty }\left\Vert vT\left(
e_{i_{1}},...,e_{i_{m}}\right) \right\Vert _{H}^{\rho }\right) ^{\frac{1}{%
\rho }}\leq \mathrm{C}_{m,r}^{\mathbb{C}}\pi _{\rho }(v)\left\Vert
T\right\Vert .
\end{equation*}
\end{proof}

Combining the definition of $GT$ spaces with the previous argument, we obtain the following result.

\begin{proposition}
\label{ell_1} Let $X$ be a $GT$ space and $H$ be a Hilbert space, $r\in
\lbrack 1,2)$ and $v:X\rightarrow H$ a linear operator. Then 
\begin{equation*}
\left( \sum_{i_{1},\dots ,i_{m}=1}^{\infty }\left\Vert vT\left(
e_{i_{1}},...,e_{i_{m}}\right) \right\Vert _{H}^{\frac{2rm}{2+\left(
m-1\right) r}}\right) ^{\frac{2+\left( m-1\right) r}{2rm}}\leq \mathrm{C}%
_{m,r}^{\mathbb{C}}\pi _{\frac{2r(m}{2+(m-1)r}}(v)\left\Vert T\right\Vert ,
\end{equation*}%
for all $m$-linear operators $T:c_{0}\times \cdots \times c_{0}\rightarrow X$%
, with $\mathrm{C}_{m,r}^{\mathbb{C}}$ as in Theorem \ref{BH_variant}.
\end{proposition}

It is worth noting that, since the optimal estimates for the constants of
Kahane's inequality are unknown (up to particular situations), the constants
in the multilinear vector-valued Bohnenblust-Hille inequality from \cite[%
Theorem 6.1]{abps.hlw} are not necessarily subpolynomial. As mentioned in 
\cite{anss}, an adaptation of the procedure presented in \cite{bps} shows
that the constants of Theorem \ref{BH_variant} are subpolynomial and,
therefore, we conclude that the estimates of Proposition \ref{ell_1} are
indeed subpolynomial. More precisely, for all $t\in \lbrack 1,2)$, there
exists a constant $\kappa _{t,\mathbb{C}}>0$ such that, for each $m>1$, 
\begin{equation*}
\mathrm{C}_{m,t}^{\mathbb{C}}\leq \kappa _{t,\mathbb{C}}\cdot m^{\frac{%
\left( \gamma -1\right) \left( t-2\right) }{2t}},
\end{equation*}%
where $\gamma $ stands for the Euler constant (see \cite[Theorem 4.4]{anss}).

%%%%%%%%%%%%%%%%%%%%%%%%%%%%%%%%%%%%%%%%%%%%%%%%%%%%%%%%%%%%%%%%%%%%

\section{A subexponential polynomial Bohnenblust-Hille inequality on Banach
lattices}

%%%%%%%%%%%%%%%%%%%%%%%%%%%%%%%%%%%%%%%%%%%%%%%%%%%%%%%%%%%%%%%%%%%%

In this section we present a vector-valued polynomial Bohnenblust-Hille
inequality on Banach lattices with estimates that recover the one presented
in \cite{bps}.

\medskip

To simplify the notation, we denote the B.f.s. $\left( X\left( \mu\right)
,\left\Vert \cdot\right\Vert _{X\left( \mu\right) }\right) $ by $\left(
X,\left\Vert \cdot\right\Vert _{X}\right) $. Let us denote by $\mu^{n}$ the
normalized Lebesgue finite measure on the $n$ -dimensional torus $\mathbb{T}%
^{n}:=\{(z_{1},\dots,z_{n})\in\mathbb{C}^{n};\,|z_{i}|=1\}$. The following
inequality is a combination of Krivine's calculus (as presented, \emph{e.g.}%
, in \cite[p.40-42]{LindTzaf}) and a result due to Bayart \cite[Theorem 9]%
{bayart} (see, for instance, \cite[p. 61]{ursula}).

\begin{theorem}[Bayart--Weissler inequality]
\label{bay_in} Let $0<p<q<\infty $. For every $M$-homogeneous polynomial $%
P(z)=\sum_{|\alpha |=m}c_{\alpha }z^{\alpha }$ on $\mathbb{C}^{n}$ with
values in a Banach lattice $Y$, 
\begin{equation*}
\left( \int_{\mathbb{T}^{n}}\left\vert P(z)\right\vert ^{q}d\mu
^{n}(z)\right) ^{\frac{1}{q}}\leq \left( \frac{q}{p}\right) ^{\frac{m}{2}%
}\left( \int_{\mathbb{T}^{n}}\left\vert P(z)\right\vert ^{p}d\mu
^{n}(z)\right) ^{\frac{1}{p}}.
\end{equation*}
\end{theorem}

When $n=1$ and $Y=\mathbb{C}$, the result is due to F. B. Weissler (\cite{weissler}) and it asserts that the optimal constant is $\sqrt{\frac{2}{p}}$. The following variant of Theorem \ref{bay_in} will be useful to obtain our main result.

\begin{lemma}
Let $1\leq p\leq2\leq q<\infty$, and $\left(X,\left\Vert\cdot\right\Vert_{X}%
\right)$ a $p$-convex and $q$-concave B.f.s.. For every $m$-homogeneous
polynomial $P(z)=\sum_{|\alpha|=m}c_{\alpha}z^{\alpha}$ on $\mathbb{C}^{n}$
with values in $X$, we have 
\begin{equation*}
\left( \sum_{|\alpha|=m}\left\Vert c_{\alpha}\right\Vert _{X}^{q}\right)^{%
\frac{1}{q}}\leq\left( \frac{2}{p}\right) ^{\frac{m}{2}}M_{q}(X)\left( \int_{%
\mathbb{T}^{n}}\left\Vert P(z)\right\Vert _{X}^{p}d\mu^{n}(z)\right) ^{\frac{%
1}{p}}.
\end{equation*}
\end{lemma}

\begin{proof}
Using the orthonormality of the monomials $z^{\alpha }$ in $L^{2}$ and the $%
q $-concavity of $X$, with $2\leq q$, we get 
\begin{equation*}
\left( \sum_{|\alpha |=m}\left\Vert c_{\alpha }\right\Vert _{X}^{q}\right) ^{%
\frac{1}{q}}\leq M_{q}(X)\left\Vert \left( \int_{\mathbb{T}%
^{n}}|P(z)|^{2}d\mu ^{n}(z)\right) ^{\frac{1}{2}}\right\Vert _{X}.
\end{equation*}%
By Theorem \ref{bay_in}, we have 
\begin{align*}
\left( \sum_{|\alpha |=m}\left\Vert c_{\alpha }\right\Vert _{X}^{q}\right) ^{%
\frac{1}{q}}& \leq M_{q}(X)\left\Vert \left( \int_{\mathbb{T}%
^{n}}|P(z)|^{2}d\mu ^{n}\left( z\right) \right) ^{\frac{1}{2}}\right\Vert
_{X} \\
& \leq M_{q}(X)\left( \frac{2}{p}\right) ^{\frac{m}{2}}\left\Vert \left(
\int_{\mathbb{T}^{n}}|P(z)|^{p}d\mu ^{n}\left( z\right) \right) ^{\frac{1}{p}%
}\right\Vert _{X} \\
& =M_{q}(X)\left( \frac{2}{p}\right) ^{\frac{m}{2}}\left\Vert \int_{\mathbb{T%
}^{n}}|P(z)|^{p}d\mu ^{n}\left( z\right) \right\Vert _{X_{\left[ p\right]
}}^{\frac{1}{p}} \\
& \leq M_{q}(X)\left( \frac{2}{p}\right) ^{\frac{m}{2}}\left( \int_{\mathbb{T%
}^{n}}\left\Vert |P(z)|^{p}\right\Vert _{X_{\left[ p\right] }}d\mu
^{n}\left( z\right) \right) ^{\frac{1}{p}} \\
& =M_{q}(X)\left( \frac{2}{p}\right) ^{\frac{m}{2}}\left( \int_{\mathbb{T}%
^{n}}\left\Vert P(z)\right\Vert _{X}^{p}d\mu ^{n}(z)\right) ^{\frac{1}{p}}.
\end{align*}
\end{proof}

Now we present the main result of this paper, whose proof is obtained by combining the arguments used in \cite[Theorem 5.2]{bps} and \cite[Theorem 5.3]{DMS}.

\begin{theorem}
\label{2convex} Let $Y$ a Banach space, $X$ a $2$-concave and $2$-convex
B.f.s., and $v:Y\rightarrow X$ an $(r,1)$-summing operator with $1\leq r<2$.
Let us fix $1\leq k\leq m$ and define 
\begin{equation*}
\rho :=\frac{2mr}{2+(m-1)r}.
\end{equation*}%
For every $m$-homogeneous polynomial $P:\ell _{\infty }^{n}\rightarrow Y$,
defined by $P(z)=\sum_{|\alpha |=m}c_{\alpha }z^{\alpha }$, we have 
\begin{equation*}
\left( \sum_{|\alpha |=m}\left\Vert vc_{\alpha }\right\Vert _{X}^{\rho
}\right) ^{\frac{1}{\rho }}\leq \left( \frac{2}{s_{k}}\right) ^{\frac{m-k}{2}%
}\left( C_{2}\left( X\right) \right) ^{k-1}\left( \prod_{j=1}^{k-1}\mathrm{K}%
_{\frac{2rj}{2+(j-1)r},2}\right) C_{m,k}M_{2}(X)\pi _{(r,1)}(v)\Vert P\Vert
_{\mathbb{D}^{n}},
\end{equation*}%
where $\mathrm{K}_{p,2}$ is the constant of the Kahane inequality, $s_{k} := 
\frac{2kr}{2+(k-1)r}$ and $\displaystyle C_{m,k}:= \frac{m^{m}}{(m-k)^{m-k}} 
\sqrt{\frac{(m-k)!}{m!}}$.
\end{theorem}

\begin{proof}
Let us fix $r<2$ and define $s_{k}<2$ by $s_{k}=\frac{2kr}{2+(k-1)r}$. Let $L: \ell_{\infty}^{n} \times\dots\times \ell_{\infty}^{n}\to Y$ be the unique $m$-linear symmetric bounded operator associated to $P$. Thus, $c_{\alpha}=|\mathbf{i}_{\alpha}|a_{\mathbf{i}_{\alpha}}$, with $a_{\mathbf{i}_{\alpha}} := L e_{\mathbf{i}_{\alpha}}$, for some $\mathbf{i}_{\alpha}\in\mathcal{J}(m,n)$. Since $|\mathbf{i}|^{\rho-1} \leq \left( |\mathbf{i}|^{\frac{1}{2}}\right)^{\rho}$, using Lemma \ref{blei.interp} with $s=s_{k}<2$, we get
\begin{align*}
\left( \sum_{|\alpha|=m}\Vert vc_{\alpha}\Vert_{X}^{\rho}\right)^{\frac {1}{\rho}}
& =\left( \sum_{\mathbf{i}\in\mathcal{J}(m,n)} 
       \Vert |\mathbf{i}|va_{\mathbf{i}} \Vert_{X}^{\rho}\right) ^{\frac{1}{\rho}} \\
& =\left( \sum_{\mathbf{i}\in\mathcal{M}(m,n)}\frac{1}{|\mathbf{i}|} \Vert|
\mathbf{i}|va_{\mathbf{i}}\Vert_{X}^{\rho}\right) ^{\frac{1}{\rho}} \\
& \leq\left( \sum_{\mathbf{i}\in\mathcal{M}(m,n)}\left\Vert |\mathbf{i}|^{%
\frac{1}{2}}va_{\mathbf{i}}\right\Vert _{X}^{\rho}\right) ^{\frac{1}{\rho }}
\\
& \leq\left[ \prod_{S\in\mathcal{P}_{k}}\left( \sum_{\mathbf{i}_{S}\in%
\mathcal{M}(k,n)}\left( \sum_{\mathbf{i}_{\widehat{S}}\in\mathcal{M}%
(m-k,n)}\left\Vert |\mathbf{i}|^{\frac{1}{2}}va_{\mathbf{i}}\right\Vert
_{X}^{2}\right) ^{\frac{s_{k}}{2}}\right) ^{\frac{1}{s_{k}}}\right] ^{\frac{1%
}{\binom{m}{k}}}.
\end{align*}

Note that $|\mathbf{i}|\leq|\mathbf{i}_{\widehat{S}}|\frac{m!}{(m-k)!}$.
Thus 
\begin{equation*}
\left( \sum_{|\alpha|=m}\Vert vc_{\alpha}\Vert_{X}^{\rho}\right)^{\frac {1}{%
\rho}} \leq \sqrt{\frac{m!}{(m-k)!}} \left[ \prod_{S\in\mathcal{P}_{k}}
\left( \sum_{\mathbf{i}_{S}\in\mathcal{M}(k,n)}\left( \sum_{\mathbf{i}_{%
\widehat{S}}\in\mathcal{M}(m-k,n)}\left\Vert |\mathbf{i}_{\widehat{S}}|^{%
\frac{1}{2}}va_{\mathbf{i}}\right\Vert _{X}^{2}\right) ^{\frac{s_{k}}{2}%
}\right) ^{\frac{1}{s_{k}}}\right] ^{\frac {1}{\binom{m}{k}}}
\end{equation*}

Let us fix $S\in\mathcal{P}_{k}(m)$. Notice that 
\begin{equation*}
P_{\mathbf{i}_{S}}(z) := \sum_{\mathbf{i}_{\widehat{S}}\in\mathcal{J}%
(m-k,n)} |\mathbf{i}_{\widehat{S}}| v a_{\mathbf{i}}z_{\mathbf{i}_{\widehat{S%
}}},
\end{equation*}
is a $(m-k)$-homogeneous polynomial on $X$, related with the symmetric $m$%
-linear operator $v L$, and 
\begin{equation*}
v L\left( e_{i_{S}},z_{i_{\widehat{S}}}\right) =P_{\mathbf{i}_{S}}(z),
\end{equation*}
where $z_{i_{\widehat{S}}}$ denotes $z$ on the coordinates $i_{\widehat{S}}$%
. Since $X$ is $2$-convex and $s_k<2$, it also is $s_k$-convex. Then Lemma
3.2 lead us to 
\begin{align*}
& \left( \sum_{\mathbf{i}_{S}\in\mathcal{M}(k,n)} \left( \sum_{\mathbf{i}_{%
\widehat{S}}\in\mathcal{M}(m-k,n)} \left\Vert |\mathbf{i}_{\widehat{S}}|^{%
\frac{1}{2}} v a_{\mathbf{i}}\right\Vert_{X}^{2} \right)^{\frac{s_{k}}{2}}
\right)^{\frac{1}{s_{k}}} \\
& = \left( \sum_{\mathbf{i}_{S}\in\mathcal{M}(k,n)} \left( \sum_{\mathbf{i}_{%
\widehat{S}} \in \mathcal{J}(m-k,n)} \left\Vert |\mathbf{i}_{\widehat{S}}| v
a_{\mathbf{i}}\right\Vert_{X}^{2} \right)^{\frac{s_{k}}{2}} \right)^{\frac{1%
}{s_{k}}} \\
& \leq \left(\frac{2}{s_{k}}\right)^{\frac{m-k}{2}} M_{2}(X) \left( \sum_{%
\mathbf{i}_{S}\in\mathcal{M}(k,n)} \left( \int_{\mathbb{T}^{n}} \left\Vert
\sum_{\mathbf{i}_{\widehat{S}}\in\mathcal{J}(m-k,n)} \left\vert\mathbf{i}_{%
\widehat{S}}\right\vert v a_{\mathbf{i}} z_{\mathbf{i}_{\widehat{S}}}
\right\Vert_X^{s_{k}}\,dz \right) \right)^{\frac{1}{s_{k}}} \\
& = \left(\frac{2}{s_{k}}\right)^{\frac{m-k}{2}} M_{2}(X) \left( \int_{%
\mathbb{T}^{n}} \left( \sum_{\mathbf{i}_{S}\in\mathcal{M}(k,n)} \left\Vert
\sum_{\mathbf{i}_{\widehat{S}}\in\mathcal{J}(m-k,n)} \left\vert\mathbf{i}_{%
\widehat{S}}\right\vert v a_{\mathbf{i}} z_{\mathbf{i}_{\widehat{S}}}
\right\Vert_X^{s_{k}} \right)\,dz \right)^{\frac{1}{s_{k}}} \\
& = \left(\frac{2}{s_{k}}\right)^{\frac{m-k}{2}} M_{2}(X) \left( \int_{%
\mathbb{T}^{n}} \left( \sum_{\mathbf{i}_{S}\in\mathcal{M}(k,n)} \left\Vert
vL\left(e_{i_{S}},z_{i_{\widehat{S}}}\right) \right\Vert_X^{s_{k}}
\right)\,dz \right)^{\frac{1}{s_{k}}}.
\end{align*}

The operator $v$ is $(r,1)$-summing and $X$ has cotype $2$. For each $z\in 
\mathbb{T}^{n}$, $L\left( \cdot ,z_{i_{\widehat{S}}}\right) $ is a $k$%
-linear operator, thus, combining \cite[Theorem 6.1 and Corollary 6.3]%
{abps.hlw} and Harris's Polarization Formula we get 
\begin{align*}
& \left( \sum_{\mathbf{i}_{S}\in \mathcal{M}(k,n)}\left\Vert vL\left(
e_{i_{S}},z_{i_{\widehat{S}}}\right) \right\Vert _{X}^{s_{k}}\right) ^{\frac{%
1}{s_{k}}} \\
& \leq \left( C_{2}\left( X\right) \right) ^{k-1}\left( \prod_{j=1}^{k-1}%
\mathrm{K}_{\frac{2rj}{2+\left( j-1\right) r},2}\right) \pi _{\left(
r,1\right) }\left( v\right) \sup_{w^{\left( 1\right) },...,w^{\left(
k\right) }\in \mathbb{T}^{n}}\left\Vert L\left( w^{\left( 1\right)
},...,w^{\left( k\right) },z_{i_{\widehat{S}}}\right) \right\Vert , \\
& \leq \left( C_{2}\left( X\right) \right) ^{k-1}\left( \prod_{j=1}^{k-1}%
\mathrm{K}_{\frac{2rj}{2+(j-1)r},2}\right) \pi _{(r,1)}(v)\frac{(m-k)!\cdot
m^{m}}{(m-k)^{m-k}\cdot m!}\Vert P\Vert _{\mathbb{D}^{n}}.
\end{align*}%
Therefore, we obtain 
\begin{align*}
& \left( \sum_{|\alpha |=m}\Vert vc_{\alpha }\Vert _{X}^{\rho }\right) ^{%
\frac{1}{\rho }} \\
& \leq \left( \frac{2}{s_{k}}\right) ^{\frac{m-k}{2}} \sqrt{\frac{(m-k)!}{m!}%
} \frac{m^{m}}{(m-k)^{m-k}}M_{2}(X)\left( C_{2}\left(X\right) \right)^{k-1}
\left( \prod_{j=1}^{k-1} K_{\frac{2rj}{2+(j-1)r},2}\right) \pi _{(r,1)}(v)
\Vert P\Vert_{\mathbb{D}^n}.
\end{align*}
\end{proof}

As was stated in the introduction, the previous result improve the hypercontractive estimate for some particular cases of Theorem \ref{dms}. In fact, notice that $\frac{2}{s_{k}}<1+\varepsilon $ for a given $\varepsilon >0$ and $k$ large enough. Thus, proceeding as in \cite{bps} we obtain the subexponential constants in Theorem \ref{2convex}.

\begin{corollary}
\label{cor_subexp} Under the same hypotheses of the previous theorem, for
any $\varepsilon>0$, there is $\kappa>0$, such that 
\begin{equation*}
\left( \sum_{|\alpha|=m}\left\Vert vc_{\alpha}\right\Vert _{X}^{\frac {2mr}{%
2+(m-1)r}}\right) ^{\frac{2+(m-1)r}{2mr}}\leq\kappa\left(
1+\varepsilon\right) ^{m}\Vert P\Vert_{\mathbb{D}^{n}},
\end{equation*}
holds for all positive integer $n$ and every $m$-homogeneous polynomial $%
P:\ell_{\infty}^{n}\rightarrow Y$ defined by $P(z)=\sum_{|\alpha|=m}c_{%
\alpha}z^{\alpha}$.
\end{corollary}

Since $M_{2}(\mathbb{C})=1$ and since the identity operator $\mathrm{id}:%
\mathbb{C}\rightarrow\mathbb{C}$ is $(1,1)$-summing with $\pi_{(1,1)}(%
\mathrm{id})=1$, taking $X=Y=\mathbb{C}$ on theorem \ref{2convex}, we
recover the polynomial Bohnenblust-Hille inequality with estimates presented
in \cite[Theorem 5.2]{bps}. More precisely, if $m\geq1$ and $%
k\in\{1,\dots,m-1\}$, then for any $n\geq1$ and any $m$-homogeneous
polynomial $P(z)=\sum_{|\alpha|=m}c_{\alpha}z^{\alpha}$ on $\mathbb{C}^{n}$, 
\begin{equation*}
\left( \sum_{|\alpha|=m}\left\vert c_{\alpha}\right\vert ^{\frac{2m}{m+1}%
}\right) ^{\frac{m+1}{2m}}\leq\left( 1+\frac{1}{k}\right) ^{\frac{m-k}{2}%
}\left( \frac{(m-k)!}{m!}\right) ^{\frac{1}{2}} \frac{m^{m}}{(m-k)^{m-k}} 
\mathrm{C}_{k}^{\mathbb{C}} \Vert P\Vert_{\mathbb{D}^{n}},
\end{equation*}
where $\mathrm{C}_{k}^{\mathbb{C}}$ denotes the $k$-th multilinear
Bohnenblust-Hille constant.

\vskip.5cm

Combining Theorem \ref{ell_1} and Theorem \ref{2convex} we have the
following result.

\begin{theorem}
\label{Hilbertlat} Let $m\geq1,\,r\in\lbrack1,2)$, $Y$ be a $GT$ space, $H$
a Hilbert function space, and $v:Y\rightarrow H$ a linear operator. Let us
fix $k\in\{1,\dots,m\}$. For all positive integers $n$ and every $m$%
-homogeneous polynomial $P:\ell_{\infty}^{n}\rightarrow Y$ defined by $%
P(z)=\sum _{|\alpha|=m}c_{\alpha}z^{\alpha}$, 
\begin{equation*}
\left( \sum_{|\alpha|=m}\left\Vert vc_{\alpha}\right\Vert _{H}^{\frac {2mr}{%
2+(m-1)r}}\right) ^{\frac{2+(m-1)r}{2mr}}\leq\left( \frac{2}{s_{k}}\right) ^{%
\frac{m-k}{2}}C_{m,k}\mathrm{C}_{k,r}^{\mathbb{C}}\pi _{\frac{2r\left(
m-1\right) }{2+\left( m-2\right) r}}(v)\Vert P\Vert_{\mathbb{D}^{n}},
\end{equation*}
where $s_{k}:=\frac{2kr}{2+(k-1)r}$ and $C_{m,k}:=\left( \frac{(m-k)!}{m!}%
\right) ^{\frac{1}{2}}\frac{m^{m}}{(m-k)^{m-k}}$.
\end{theorem}

\noindent \textbf{Acknowledgment.} The authors would like to thank Prof. E. A. S\'{a}nchez P\'{e}rez for his fruitful comments regarding Banach lattices. Also the authors would like to thank the reviewer for insightful remarks and suggestions that improved the presentation of this paper.

\end{document}